\documentclass{llncs}
\usepackage[utf8]{inputenc}
\usepackage{amsmath}
\usepackage{amsfonts}
\usepackage{amssymb}
\usepackage{bussproofs}
\usepackage{footmisc}
\usepackage{mathtools}
\usepackage{bussproofs}
\usepackage{relsize}
\usepackage{scalerel}
\usepackage[bbgreekl]{mathbbol}
\usepackage{lineno}
\usepackage{etex}
\usepackage{xcolor}
\usepackage{tikz}
\usepackage{tkz-graph}

\newcommand{\supr}{{\rm sup}}
\newcommand{\infi}{{\rm inf}}
\newcommand{\GF}{{\rm G}_{[0,1]}^{\rm QF}}
\newcommand{\MF}{{\rm MC}^{\rm QF}}
\newcommand{\M}{{\rm MC}}

\newcommand{\Pb}{{\rm {\bf P}}}
\newcommand{\F}{{\rm {\bf F}}}
\newcommand{\ar}{{\rm {\bf ar}}}
\newcommand{\Ab}{{\rm {\bf A}}}
\newcommand{\Sb}{{\rm {\bf S}}}

\makeatother

\author{Matthias Baaz\inst{1} and Anela Lolic\inst{2}}
\authorrunning{M. Baaz and A. Lolic}

\institute{
  Institut f\"{u}r Diskrete Mathematik und Geometrie 104\\
   Technische Universit\"at Wien\\
    Vienna, Austria\\
\email{baaz@logic.at}
\and
   Institut f\"{u}r Diskrete Mathematik und Geometrie 104\\
   Technische Universit\"at Wien\\
    Vienna, Austria\\
    \email{anela@logic.at}
 }
 
\title{First-Order Interpolation Derived from Propositional Interpolation}

\tikzset
{
    treenode/.style = {circle, draw=black, align=center},
    subtree/.style  = {isosceles triangle, draw=black, align=center, minimum height=0.5cm, minimum width=1cm, shape border rotate=90, anchor=north}
}
\begin{document}
\maketitle

\begin{abstract}
This paper develops a general methodology to connect propositional and first-order interpolation. In fact, the existence of suitable skolemizations and of Herbrand expansions together with a propositional interpolant suffice to construct a first-order interpolant. This methodology is realized for lattice-based finitely-valued logics, the top element representing true. It is shown that interpolation is decidable for these logics.
\end{abstract}


\section{Introduction}
Ever since Craig's seminal paper on interpolation \cite{craig1957three}, interpolation properties have been recognized as important properties of logical systems. 
Recall that a logic $L$ has \emph{interpolation} if whenever $A \to B$ holds in $L$ there exists a formula $I$ in the common language of $A$ and $B$ such that $A \to I$ and $I \to B$ hold in $L$.

Propositional interpolation properties can be determined and classified with relative ease using the ground-breaking results of Maksimova cf. \cite{maksimova1979interpolation,maksimova1977craig,maksimova2000intuitionistic}. This approach is based on an algebraic analysis of the logic in question. In contrast first-order interpolation properties are notoriously hard to determine, even for logics where propositional interpolation is more or less obvious. For example it is unknown whether $\GF$ (first-order infinitely-valued G\"odel logic) interpolates (cf \cite{aguilera2017ten}) and even for $\MF$, the logic of constant domain Kripke frames of $3$ worlds with $2$ top worlds (an extension of MC), interpolation proofs are very hard cf. Ono \cite{ono1983model}. This situation is due to the lack of an adequate algebraization of non-classical first-order logics.

In this paper we present a proof theoretic methodology to reduce first-order interpolation to propositional interpolation:
\[ \left. \begin{array}{r} \mbox{existence of suitable skolemizations } + \\\mbox{existence of Herbrand expansions } + \\\mbox{propositional interpolance } \, \, \, \, \end{array} \right\} \to \begin{array}{c} \mbox{first-order}\\\mbox{interpolation.}\end{array} \]

The construction of the first-order interpolant from the propositional interpolant follows this procedure:
\begin{enumerate}
\item Develop a validity equivalent skolemization replacing all strong quantifiers (negative existential or positive universal quantifiers) in the valid formula $A \supset B$ to obtain the valid formula $A_1 \supset B_1$.

\item Construct a valid Herbrand expansion $A_2 \supset B_2$ for $A_1 \supset B_1$. Occurrences of $\exists x B(x)$ and $\forall x A(x)$ are replaced by suitable finite disjunctions $\bigvee B(t_i )$ and conjunctions $\bigwedge B(t_i )$, respectively.

\item Interpolate the propositionally valid formula $A_2 \supset B_2$ with the propositional interpolant $I^{*}$:
\[ A_2 \supset I^* \quad \mbox{and} \quad I^* \supset B_2 \]
are propositionally valid.

\item Reintroduce weak quantifiers to obtain valid formulas
\[ A_1 \supset I^* \quad \mbox{and} \quad I^* \supset B_1 . \]

\item Eliminate all function symbols and constants not in the common language of $A_1$ and $B_1$ by introducing suitable quantifiers in $I^*$ (note that no Skolem functions are in the common language, therefore they are eliminated). Let $I$ be the result.

\item $I$ is an interpolant for $A_1 \supset B_1$. $A_1 \supset I$ and $I \supset B_1$ are skolemizations of $A \supset I$ and $I \supset B$. Therefore $I$ is an interpolant of $A \supset B$.
\end{enumerate}

We apply this methodology to lattice based finitely-valued logics and the weak quantifier and subprenex fragments of infinitely-valued first-order G\"odel logic.

Note that finitely-valued first-order logics admit variants of Maehara's \\
Lemma and therefore interpolate if all truth values are quantifier free definable \cite{miyama1974interpolation}. For logics where not all truth-values are represented by quantifier-free formulas this argument does not hold, which explains the necessity of different interpolation arguments for e.g. $\MF$ (the result for $\MF$ is covered by our framework, cf. Example \ref{ex.running}). We provide a decision algorithm for the interpolation property for lattice based finitely-valued logics. 

Most results in interpolation are concerned with the question whether a given logic interpolates but not with the more general question, to check the minimal extensions with that property. Our framework allows for the calculation of the relevant first-order extensions, which is given by the calculation of the relevant propositional extensions. For classical logic we show in this way that the fragment with $\top, \land, \lor, \forall, \exists, \supset$ interpolates, see Example \ref{ex.10}.

\section{Lattice Based Finitely-Valued Logics}\label{sec.lattice}

\begin{definition}[signature, cf \cite{DBLP:journals/tcs/CintulaDM19}]
A signature with polarities $\mathcal{L}^\to$ (or simply signature) consists of a finite set $C_{\mathcal{L}^\to}$ of symbols (called connectives), where each connective $c$ has an assigned arity $n_c \in \mathbb{N}$ and polarity $p_c$: $\{1, 2, \ldots , n_c \} \to \{-, +\}$. 
It is called lattice-oriented if $C_{\mathcal{L}^\to}$ contains three binary connectives $\lor$, $\land$ and $\to$ with $p_{\lor}(i) = p_{\land}(i) = +$ for $i \in \{1, 2\}$ and $p_{\to}(1) = -$ and $p_{\to}(2) = +$.
\end{definition}

\begin{definition}[$\mathcal{L}^\to$-lattice] \label{def.latticecond}
Given any lattice-oriented signature $\mathcal{L}^\to$, a finite $\mathcal{L}^\to$-lattice is an algebraic structure $\langle L, \{c^L\}_{c \in C_{\mathcal{L}^\to}} \rangle$ satisfying
\begin{enumerate}
	\item $\langle L, \lor^L, \land^L \rangle$ is a lattice with an order defined by $x \leq^L y \leftrightarrow x \land^L y = x$.
	\item $c^L$ is an $n_c$-ary operator on $L$ for each $c \in C_{\mathcal{L}^\to}$ such that for $1 \leq i \leq n_c$,
	\begin{enumerate}
		\item if $p_c(i) = +$, then $c^L$ is monotone in the $i$-th argument
		\item if $p_c(i) = -$, then $c^L$ is antitone in the $i$-th argument.
	\end{enumerate}
	\item $1 \leq^L A \to B$ iff $A \leq^L B$.
\end{enumerate}
\end{definition}
We abbreviate $A \to^L B \land^L B \to^L A$ with $A \leftrightarrow^L B$. We write $\models_0 A$ for $1 \leq A$ for all elements ($A$ is valid) and $A_1, \ldots , A_n \models_0 A'$ for $\models_0 A_1 \land^L \ldots \land^L A_n \to^L A'$. The logic $\mathbf{L^0}(\mathcal{L}^\to)$ is defined as the set of valid sentences $A$. The monotony and antitony of connectives is iterated as usual. 

\begin{proposition}\label{prop.b}
For all logics the following hold
\begin{enumerate}
\item $\models A \to A$,
\item If $\models B$ then $\models A \to B$,
\item If $\models A \to B$ and $\models C \to D$ then $\models (B \to C) \supset (A \to D)$.
\end{enumerate}
\end{proposition}
Let $L' \subseteq L$ and $\mathcal{D}_{L'} = \{c_i \ | \ c_i \mbox{ constant with value } i \in L'\}$. $\mathcal{L}^\to(\mathcal{D}_{L'})$ is a lattice with extended signature by $\mathcal{D}_{L'}$ where $c_i$ has value $i$.

We will omit $L$ from the connectives and $\leq$ when the semantical context is obvious.
\begin{definition} 
Let $L$ be the domain of $\mathcal{L}^\to$. A function $f: L^n \to L$ is representable if there is a word $w(x_1, \ldots, x_n)$ such that $w(d_1, \ldots, d_n)$ evaluates to $d$ if $f(d_1, \ldots, d_n) = d$. 

Let $V(\mathcal{L}^\to)$ be the set of values of constant functions of $\mathcal{L}^\to$.
\end{definition}

\begin{proposition}\label{prop.c}
\mbox{ }
\begin{enumerate}
	\item If $V(\mathcal{L}^\to) = \emptyset$ then $\mathcal{L}^\to$ does not admit the interpolation property.
	\item If $V(\mathcal{L}^\to) = L$ then $\mathcal{L}^\to$ admits the interpolation property.
\end{enumerate}
\end{proposition}

\begin{proof}
\mbox{ }
\begin{enumerate}
	\item $x \leq (y \to y)$ has as only possible interpolant a closed word denoting $1$. 
	\item Consider $a(x_1, \ldots, x_n) \leq b$ valid with left variables $x_1, \ldots , x_n$.\\
	\[ I = \bigvee_{\langle v_{i_1}, \ldots v_{i_n} \rangle \in V(\mathcal{L}^\to) \times F(\mathcal{L}^\to)} a(v_{i_1}, \ldots , v_{i_n}) \]
	is an interpolant as $a(x_1 , \ldots ,x_n) \leq I$ and $I \leq b$.
\end{enumerate}
\end{proof}
Suppose that $\langle L, \lor, \land, \to, \&, 0, 1\rangle$ is an $\mathcal{L}^\to$ lattice. If $\langle L, \&, \overline{1}\rangle$ is a commutative monoid and $\to$ is the residuum of $\&$ ($x \& y \leq z \leftrightarrow x \leq y \to z$ for all $x,y,z \in L$), then $L$ is a commutative pointed residual lattice. Note that $p_{\&}(i) = +$ for $i \in \{1, 2\}$.
\begin{proposition}
The condition on implication given by Definition \ref{def.latticecond} $3$ implies the definition of implication by residuation.
\end{proposition}
\begin{proof}
The definition of implication by residuation implies Definition \ref{def.latticecond} $3$:
$$ 1 \leq A \to B \quad \Rightarrow \quad 1 \& A \leq B \quad \Rightarrow \quad A \leq B,$$
$$ A \leq B \quad \Rightarrow \quad 1 \& A \leq B \quad \Rightarrow \quad 1 \leq A \to B.$$
\end{proof}
\begin{remark}
Note that the condition on implication given by Definition \ref{def.latticecond} $3$ is more general, i.e. there are implications fulfilling (iii) which do  not correspond to a residuation of any monoid.
\end{remark}
\begin{example}
Consider $\langle \{0,u_1, u_2, 1\}, \lor, \land, \to \rangle$ given by

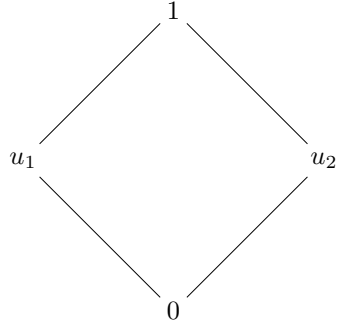
\begin{center}
\begin{tikzpicture}
\GraphInit[vstyle=Empty]
\Vertex[L=$1$,x=1,y=-2]{X}
\Vertex[L=$u_1$,x=-1,y=-4]{A}
\Vertex[L=$u_2$,x=3,y=-4]{B}
\Vertex[L=$0$,x=1,y=-6]{D}
\Edges[](A,X,B)
\Edges[](A,D,B)
\end{tikzpicture}
\end{center}
with 
\[ u \to v = \begin{cases} 1 \quad u \leq v \\ 0 \quad u \not\leq v \end{cases}\]
Then $\to$ is not given by residuation.

Assume on the contrary there were a monoid with $\&$ and $1$ on $\{u_1, u_2, 1\}$
\[\begin{array}{lclclclc}
u_1 \& u_2 = 0 & \Rightarrow & u_1 \leq u_2 \to 0 & \Rightarrow & u_1 \leq 0 & & & \mbox{ contradiction.} \\
u_1 \& u_2 = u_1 &\Rightarrow & u_1 \leq u_2 \to u_1 &\Rightarrow& u_1 \leq 0 & & & \mbox{ contradiction.}\\
u_1 \& u_2 = u_2 & \Rightarrow & u_2 \leq u_1 \to u_2 &\Rightarrow& u_2 \leq 0  & & & \mbox{ contradiction.}\\
u_1 \& u_2 = 1 & \Rightarrow & 1 \leq u_1 \& u_2, u_2 \leq 1 &\Rightarrow& 1 \leq u_1 \& 1 &\Rightarrow& 1 \leq u_1 & \mbox{ contradiction.}\\
\end{array}\]
\end{example}

\begin{proposition}
$F(\mathcal{L}^\to)$ is decidable. 
\end{proposition}
\begin{proof}
Use the following algorithm, which determines all functions with $\leq n$ variables for a finite set of finitely-valued matrices of $\leq m$ values:\\[1ex]
Level $0$: Start with the $n$ input columns of the lattice for an $n$-placed connective and constant columns for the constants.\\[1ex]
Level $n+1$: Apply the connectives to all existing columns in all possible ways and add columns if a new column occurs.\\[1ex]
This algorithm terminates in $\leq m^n$ rounds.
\end{proof}

\begin{example}
Values $= \{0, \frac{1}{2}, 1\}$
\setlength{\tabcolsep}{8pt}

\begin{minipage}{0.4\linewidth}
\begin{tabular}{ c | c c c }
$\to$ & $0$ & $\frac{1}{2}$ & $1$ \\[1ex]
\hline
$0$ & $1$ & $1$ & $1$  \\[1ex]
$\frac{1}{2}$ & $\frac{1}{2}$ & $1$ & $1$  \\[1ex]
$1$ & $0$ & $\frac{1}{2}$ & $1$
\end{tabular}
\end{minipage}
\begin{minipage}{0.4\linewidth}
\begin{tabular}{ c | c }
 & $\overline{0}$ \\[1ex]
\hline
$0$ & $0$ \\[1ex]
$\frac{1}{2}$ & $0$ \\[1ex]
$1$ & $0$
\end{tabular}
\end{minipage}

\bigskip

$0$ variables

\small
\begin{minipage}{0.45\linewidth}
Level $0$: $\left( \begin{array}{c}
0 \\                                              
0 \\
0                                         
\end{array}\right)$
\end{minipage}
\begin{minipage}{0.5\linewidth}
Level $1$: Level $0$ +
$\left( \begin{array}{c}
1 \\                                              
1 \\
1                                         
\end{array}\right)$
\end{minipage}
\normalsize

\bigskip

$\leq 1$ variables

\small
\begin{minipage}{0.45\linewidth}
Level $0$: 
$\left( \begin{array}{c}
0 \\                                              
\frac{1}{2} \\
1                                              
\end{array}\right)$ $\left( \begin{array}{c}
0 \\                                              
0 \\
0                                         
\end{array}\right)$
\end{minipage}
\begin{minipage}{0.5\linewidth}
Level $1$:
Level $0$ + 
$\left( \begin{array}{c}
1 \\                                              
1 \\
1                                              
\end{array}\right)$ $\left( \begin{array}{c}
1 \\                                              
\frac{1}{2} \\
0                                         
\end{array}\right)$
\end{minipage}

\begin{minipage}{0.45\linewidth}
Level $2$:
Level $1$ + 
$\left( \begin{array}{c}
1 \\                                              
1 \\
0                                              
\end{array}\right)$ $\left( \begin{array}{c}
0 \\                                              
1 \\
1                                         
\end{array}\right)$
\end{minipage}
\begin{minipage}{0.5\linewidth}
Level $3$:
Level $2$ + 
$\left( \begin{array}{c}
0 \\                                              
0 \\
1                                              
\end{array}\right)$ $\left( \begin{array}{c}
1 \\                                              
0 \\
0                                         
\end{array}\right)$ $\left( \begin{array}{c}
1 \\                                              
\frac{1}{2} \\
1                                         
\end{array}\right)$
\end{minipage}

\begin{minipage}{0.45\linewidth}
Level $4$:
Level $3$ + 
$\left( \begin{array}{c}
0 \\                                              
\frac{1}{2} \\
0                                              
\end{array}\right)$ $\left( \begin{array}{c}
1 \\                                              
0 \\
1                                         
\end{array}\right)$
\end{minipage}
\begin{minipage}{0.5\linewidth}
Level $5$:
Level $4$ + 
$\left( \begin{array}{c}
0 \\                                              
1 \\
0                                              
\end{array}\right)$
\end{minipage}

All functions $f(x)$ are representable, where
$$\langle f(0), f(\frac{1}{2}), f(1) \rangle \leq \langle \{0,1\}, \{0, \frac{1}{2}, 1\}, \{0,1\} \rangle. $$
\end{example}

\section{Interpolation for Finitely-Valued Lattice Based Logics is Decidable}

\begin{definition} $\mathcal{L}^\to$ has the interpolation property iff $a(\overline{x}, \overline{y}) \leq b(\overline{y}, \overline{z})$ valid implies $a(\overline{x}, \overline{y}) \leq i(\overline{y})$ and $i(\overline{y}) \leq b(\overline{y}, \overline{z})$ for an interpolant $i(\overline{y})$ (all variables are indicated). We call the variables occurring only in $a$ left variables, the variables occurring only in $b$ the right variables and the variables occurring in $a$ and in $b$ the intersection variables.

$\mathcal{L}^\to$ admits the Craig interpolation property $1 \leq a(\overline{x}, \overline{y}) \to b(\overline{y}, \overline{z})$ valid implies $1 \leq a(\overline{x}, \overline{y}) \to i(\overline{y})$ and $1 \leq i(\overline{y}) \to b(\overline{y}, \overline{z})$ for an interpolant $i(\overline{y})$ (all variables are indicated).
\end{definition}

\begin{proposition} A logic based on $\mathcal{L}^{\to}$ interpolates iff  $\mathcal{L}^{\to}$ interpolates.
\end{proposition} 

\begin{example}
$\mathcal{L}^\to = \langle \{0,1,a\}, \lor, \land, \to, \overline{0}, \overline{1} \rangle$, $\overline{0} = 0$, $\overline{1} = 1$, $0 < a$ and $a < 1$
\[ u \to v = \begin{cases} 1 \quad u \leq v \\ 0 \quad u = 1 \mbox{ and } v = 0 \\ a \quad \mbox{else} \end{cases}\]
$\mathbf{L^0} (\mathcal{L}^\to)$ does not interpolate as 
\[ \models^0 (x \land (x \to \overline{0})) \to (y \lor (y \to \overline{0})) \]
does not admit an interpolant, as the only possible interpolant is a constant with value $a$ (there are no common variables in the antecedent and the succedent).

Let $\mathcal{L}^\to = \langle \{0,1,a\}, \lor, \land, \to, \overline{0}, \overline{a} \rangle$, $\overline{0} = 0$, $\overline{a} = a$, $0 < a$ and $a < 1$
$\textbf{L}^0 (\mathcal{L}^\to)$ interpolates as all truth constants are representable, $1$ by $\overline{0} \to \overline{0}$ (c.f. Section \ref{maintheorem}).
\end{example}

\begin{example} \label{ex.running}
Finite propositional and constant-domain Kripke frames can be understood as lattice-based finitely valued logics: Consider upwards closed subsets $\Gamma \subseteq W$, $W$ is the set of worlds, and order them by inclusion. A formula $A$ is assigned the truth value $\Gamma$ iff $A$ is true at exactly the worlds in $\Gamma$.

The constant-domain intuitionistic Kripke frame $\mathcal{K}$ in Fig. \ref{constantdomainKripkeFrame} is represented by the lattice $\mathcal{L}^\to(\{ {1 \quad 1 \choose 1}, {1 \quad 1 \choose 0}, {0 \quad 1 \choose 0}, {1 \quad 0 \choose 0}, {0 \quad 0 \choose 0} \}, \lor, \land, \to,$ $\overline{ 0 \quad 0 \choose 0}$ $)$ in Fig. \ref{lattice}.

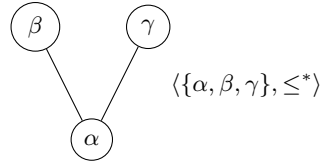
\begin{figure}
\centering
\begin{tikzpicture}
\node(z)[treenode]{$\alpha$}[grow=up]
  child { node[treenode]{$\gamma$} edge from parent node[left,draw=none] {$$} }
  child { node[treenode]{$\beta$} edge from parent node[right,draw=none] {$\qquad \qquad \langle \{\alpha, \beta, \gamma \}, \leq^* \rangle$}};
\end{tikzpicture}

\caption{Constant-domain intuitionistic Kripke frame $\mathcal{K}$.}
\label{constantdomainKripkeFrame}
\end{figure}
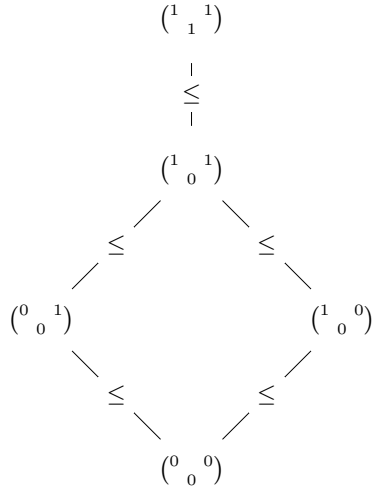
\begin{figure}
\centering
\begin{tikzpicture}
\GraphInit[vstyle=Empty]
\Vertex[L=$ 1 \quad 1 \choose 1$,x=1,y=0]{Y};
\Vertex[L=$ 1 \quad 1 \choose 0$,x=1,y=-2]{X}
\Vertex[L=$ 0 \quad 1 \choose 0$,x=-1,y=-4]{A}
\Vertex[L=$ 1 \quad 0 \choose 0$,x=3,y=-4]{B}
\Vertex[L=$ 0 \quad 0 \choose 0$,x=1,y=-6]{D}
\Edges[label=$\leq$](Y,X) 
\Edges[label=$\leq$](A,X,B)
\Edges[label=$\leq$](A,D,B)
\end{tikzpicture}
\caption{The lattice.}
\label{lattice}
\end{figure}

where
\[ u \to v = \begin{cases} 1 \quad u \leq v \\ v \quad \mbox{else} \end{cases}\]
$\M = \textbf{L}^0(\mathcal{L}^\to)$ is the set of valid propositional sentences.
\end{example}
Propositional interpolation is easily demonstrated for MC, one of the seven intermediate logics which admit propositional interpolation \cite{maksimova1977craig}. Previous proofs for the interpolation of $\MF$, the first-order variant of MC, are quite involved, \cite{ono1983model}. In fact, in Section \ref{maintheorem}, Example \ref{ex.interpolation} we will show that this interpolation result is a corollary of the main theorem of this paper.

\begin{proposition} \label{prop.two}
$$ (A(x_1, \ldots , x_n) \land \bigwedge_{i=1}^n x_i \leftrightarrow x'_i) \leftrightarrow A(x'_1, \ldots , x'_n) \land \bigwedge_{i=1}^n x_i \leftrightarrow x'_i. $$
\end{proposition}
\begin{proof}
By induction on the complexity of $A$.
\end{proof}

\begin{theorem} \label{th.one}
It is decidable if a given finite $\mathcal{L}^\to$ admits the interpolation property. 
\end{theorem}
This theorem follows from the following three lemmas.

\begin{definition}
Let $\pi^n$ be a partition of $X$ with $\leq n$ equivalence classes $E_i$ and let $x^{E_i}$ be a representative of $E_i$. $\sigma_{\pi^n}: X \to X$ be defined by $\sigma(x) = x^{E_i}$ for $x \in E_i$. Let $\Sigma^n(x)$ be the set of all such substitutions. 
\end{definition}

\begin{lemma}\label{lem.1}
$\mathcal{L}^\to$ with domain $A$, where $|A| \leq n$ admits the interpolation property iff it admits the interpolation property for all $a \leq b$ where the number of left variables is $\leq n$. 
\end{lemma}
\begin{proof}
Let $a \leq b$ be valid and let $X$ be the set of left variables. Then $a \sigma \leq b \sigma$ is valid and consequently $a \sigma \leq b$ is valid for all $\sigma \in \Sigma^n(X)$. Then $(\bigvee{\sigma \in \Sigma^n(X)} a) \leq b$ is valid. As $a \leq (\bigvee{\sigma \in \Sigma^n(X)} a)$ is valid because the number of classes of variables identified by any valuation is $\leq n$. Therefore the interpolant for $(\bigvee{\sigma \in \Sigma^n(X)} a) \leq b$ is an interpolant for $a \leq b$.
\end{proof}

\begin{lemma}\label{lem.2}
$\mathcal{L}^\to$ with domain $A$, where $|A| \leq n$ admits the interpolation property iff it admits the interpolation property for all $a \leq b$ where the number of right variables is $\leq n$. 
\end{lemma}
\begin{proof}
Let $a \leq b$ be valid and let $X$ be the set of right variables. Then $a \sigma \leq b \sigma$ is valid and consequently $a \leq b \sigma$ is valid for all $\sigma \in \Sigma^n(X)$. Then $a \leq (\bigwedge_{\sigma \in \Sigma^n(X)} b)$ is valid. As $(\bigwedge_{\sigma \in \Sigma^n(X)} b) \leq b$ is valid because the number of classes of variables identified by any valuation is $\leq n$. Therefore the interpolant for $a \leq (\bigwedge_{\sigma \in \Sigma^n(X)} b)$ is an interpolant for $a \leq b$.
\end{proof}

\begin{lemma}\label{lem.3}
$\mathcal{L}^\to$ with domain $A$, where $|A| \leq n$ admits the interpolation property iff it admits the interpolation property for all $a \leq b$ where the number of intersection variables is $\leq n$. 
\end{lemma}
\begin{proof}
Let $a \leq b$ be valid and let $X$ be the set of intersection variables. For $\sigma \in \Sigma^n(X)$ let $C_{\sigma}$ be $\bigwedge_{x \in X} (x\sigma \to x) \land (x \to x\sigma)$.
Then 
$$a \leq \bigvee_{\sigma \in \Sigma^n(X)} (a \land C_{\sigma})$$
as under any valuation at least one of $C_{\sigma}$ is evaluated to $1$.
Therefore 
$$a \leq \bigvee_{\sigma \in \Sigma^n(X)} (a\sigma \land C_{\sigma})$$
by Proposition \ref{prop.two}.
Now $a \sigma \leq b \sigma$ for all $\sigma \in \Sigma^n(X)$, $|X\sigma| \leq n$.
In case there is always an interpolant $I_\sigma$ we obtain
$$\bigvee_{\sigma \in \Sigma^n(X)} (a\sigma \land C_{\sigma}) \leq \bigvee_{\sigma \in \Sigma^n(X)} (I_\sigma \land C_{\sigma})$$
and 
$$\bigvee_{\sigma \in \Sigma^n(X)} (I_\sigma \land C_{\sigma}) \leq \bigvee_{\sigma \in \Sigma^n(X)} (b \sigma \land C_{\sigma})$$ 
as $a\sigma \leq I_\sigma$ and $I_\sigma \leq b \sigma$ for all $\sigma \in \Sigma^n(X)$.
$$\bigvee_{\sigma \in \Sigma^n(X)} (b \sigma \land C_{\sigma}) \leq \bigvee_{\sigma \in \Sigma^n(X)} (b \land C_{\sigma})$$
again by Proposition \ref{prop.two}.

Finally, $\bigvee_{\sigma \in \Sigma^n(X)} (b \land C_{\sigma}) \leq b$. Therefore,  $\bigvee_{\sigma \in \Sigma^n(X)} I_\sigma \land C_\sigma$ is a suitable interpolant for $a \leq b$.
\end{proof}

\begin{proof}[Proof of Theorem \ref{th.one}]
By Lemma \ref{lem.1}, \ref{lem.2} and \ref{lem.3} the number of left variables, right variables and intersection variables is bound by $n$. Consider all pairs of words denoting the representable functions with the limitation of variable occurrences as above. $(a,b)$: check whether $a \leq b$ is valid. In case it is valid check whether there is a representable function whose representation might serve as interpolant.
\end{proof}

\begin{corollary}
It is decidable if a given finite $\mathcal{L}^\to$ admits the Craig interpolation property. 
\end{corollary}
\begin{example} Consider $\mathcal{L}^\to = \langle \{0, \frac{1}{2}, 1\}, \lor, \land, \to, \overline{0}\rangle$ with $0 < 1/2$ and $1/2 < 1$.
$$ a \to b = \begin{cases} 1 \quad \quad \quad a \leq b \\ b-a \quad \mbox{ else} \end{cases}$$
($b-a$ in the usual sense). $x \land (x \to 0) \to y \lor (y \to 0)$ interpolates iff a constant for $\frac{1}{2}$ is added.
\end{example}
\begin{remark}
Note that both extension and reduction of the signature may influence interpolation. 
\end{remark}
\begin{example} Consider $\mathcal{L}^\to = \langle \{0, \frac{1}{2}, 1\}, \lor, \land, \to, \overline{0}\rangle$, where $\overline{0} = 0$, $0 < \frac{1}{2}$ and $\frac{1}{2} < 1$
$$ a \to b = \begin{cases} 1 \quad a \leq b \\ b \quad \mbox{else} \end{cases}$$
$\mathbf{L^0}(\mathcal{L}^\to)$ interpolates ($\mathbf{L^0}(\mathcal{L}^\to)$ is a three-valued G\"odel logic).

Extend $\mathcal{L}^\to$ to $\mathcal{L'}^\to  = \langle \{0, \frac{1}{2}, 1\}, \lor, \land, \to, \Box_{up}, \Box_{down}, \overline{0}\rangle$.  $\mathcal{L'}^\to$ is defined as before, with the addition of 
$$ \Box_{up}(i) = \begin{cases} i \quad i \leq \frac{1}{2} \\ \frac{1}{2} \quad \mbox{else} \end{cases}$$
$$ \Box_{down}(i) = \begin{cases} i \quad \frac{1}{2} \leq i \\ \frac{1}{2} \quad \mbox{else} \end{cases}$$
$\mathbf{L^0}(\mathcal{L'}^\to)$ does not interpolate as $\Box_{down}(x) \to \Box_{up}(y)$ does not interpolate. 
Note that the addition of constants alone does not weaken the interpolation property.
\end{example}
Proposition \ref{prop.c} $2.$ makes it possible to characterize all extensions of a lattice by constants which admit interpolation. $\mbox{SPECTRUM}(\mathcal{L}^\to) = $
\[ \{V \, | \, \mathcal{L}^\to \mbox{ extended by constants representing the values in } V \mbox{ interpolates}\}. \]
\begin{corollary}
{\rm SPECTRUM} is a calculable function.
\end{corollary}

\section{First-Order Logic}

\begin{definition}[predicate language]
A (countable) predicate language $\mathcal{P}$ is a triple $\langle \Pb, \F, \ar \rangle$ consisting of disjoint countable sets $\Pb$ and $\F$ of predicate and function symbols, and a function $\ar \colon \Pb \cup \F \to \mathbb{N}$ assigning arities to these symbols. We call nullary function symbols object constants and nullary predicate symbols propositional atoms. For convenience, a predicate language containing only propositional atoms will be called propositional.
\end{definition}
Let us fix a lattice-oriented signature $\mathcal{L}^\to$ and a predicate language $\mathcal{P} = \langle \Pb, \F, \ar \rangle$. We define $\mathcal{P}$-terms, atomic $\mathcal{P}$-formulas, and $\langle \mathcal{L}, \mathcal{P}\rangle$-formulas as in classical logic using a fixed countably infinite set $OV$ of object variables $x, y, \ldots$, the quantifiers $\forall$ and $\exists$ and the connectives in $\mathcal{L}^\to$.
$\langle \mathcal{L}, \mathcal{P} \rangle$-formulas are denoted with $\varphi, \psi, \ldots$.
The notions of bound and free variables, closed terms, sentences, prenex formulas, and substitutability in formulas are defined in the standard way.

An $\langle \mathcal{L}^\to, \mathcal{P}\rangle$-structure $\mathcal{S}$ is a pair $\langle \Ab, \Sb \rangle$ such that
\begin{enumerate}
	\item $\Ab$ is a finite $\mathcal{L}^\to$-lattice,
	\item $\Sb$ is a triple $\langle S, \{P^{\Sb}\}_{P \in \Pb}, \{f^{\Sb}\}_{f \in \F}\rangle$ where
	\begin{itemize}
		\item $S$ is a non-empty set,
		\item $P^\Sb \colon S^n \to A$ is a function for each $n$-ary predicate symbol $P \in \Pb$,
		\item $f^\Sb \colon S^n \to S$ is a function for each $n$-ary function symbol $f \in \F$. 
	\end{itemize}	 
\end{enumerate}
An $\mathcal{S}$-evaluation is a mapping $v \colon OV \to S$. For any $\mathcal{S}$-evaluation $v$ we denote by $v[x \to a]$ the $\mathcal{S}$-evaluation satisfying $v[x \to a](x) = a$ and $v[x \to a](y) = v(y)$ for each $y \not= x$. Terms and formulas are evaluated in $\mathcal{S}$ with respect to an $\mathcal{S}$-evaluation $v$ according to the following conditions, where $f \in \F, P \in Pb$ and $c \in C_{\mathcal{L}^\to}$:
\begin{itemize}
	\item $|| x ||_v^\mathcal{S} = v(x)$,
	\item $|| f(t_1, \ldots, t_n) ||_v^\mathcal{S} = f^\mathcal{S}(||t_1||_v^\mathcal{S}, \ldots, || t_n ||_v^\mathcal{S})$,
	\item $|| P(t_1, \ldots, t_n) ||_v^\mathcal{S} = P^\mathcal{S}(||t_1||_v^\mathcal{S}, \ldots, || t_n ||_v^\mathcal{S})$,
	\item $|| c(\varphi_1, \ldots, \varphi_n) ||_v^\mathcal{S} = c^\Ab(||\varphi_1||_v^\mathcal{S}, \ldots, || \varphi_n ||_v^\mathcal{S})$,
	\item $|| \forall x \varphi ||_v^\mathcal{S} = \bigwedge \{ || \varphi||_{v[x \to a]}^\mathcal{S} \ | \ a \in S\}$,
	\item $|| \exists x \varphi ||_v^\mathcal{S} = \bigvee \{ || \varphi ||_{v[x \to a]}^\mathcal{S} \ | \ a \in S\}$.
\end{itemize}
We write $\models_1 C$ ($C$ is valid in $\mathcal{L}^\to$) iff for every structure $\langle \Ab, \Sb \rangle$ and every $v$ $1 \leq ||C||^\Sb_v$ and $C_1, \ldots , C_n \models_1 A'$ for $\models_1 \forall \overline{x} \bigwedge_{i=1}^n C_i \to C'$. The first-order logic $\mathbf{L^1}(\mathcal{L}^\to)$ is defined as the set of valid sentences $C$ in $\mathcal{L}^\to$.

Monotony and antitony of first-order contexts are iterated as usual, quantifiers do not change the polarity.
\begin{lemma}\label{alpha}
For formulas $A$, $B$ and a corresponding context $C( \circ )$ it holds
\[ \mbox{ if } \quad \models_1 A \to B \quad \mbox{ then } \quad \models_1 C(A) \to C(B) \]
if $\circ$ occurs positively and 
\[ \mbox{ if } \quad \models_1 A \to B \quad \mbox{ then } \quad \models_1 C(B) \to C(A) \]
if $\circ$ occurs negatively.
\end{lemma}
\begin{proof}
By iteration of the polarity of the connectives.
\end{proof}

\begin{definition}[weak interpolation property]
$\mathbf{L^1}(\mathcal{L}^\to)$ has the weak interpolation property if for every $\models_1 A \to B$ there is a $I$ with predicate symbols occurring in both $A$ and $B$ such that $\models_1 A \to I$ and $\models_1 I \to B$ hold.
\end{definition}

\begin{definition}[strong interpolation property]
$\mathbf{L^1}(\mathcal{L}^\to)$ has the strong interpolation property if for every $\models_1 A \to B$ there is a $I$ with predicate and function symbols occurring in both $A$ and $B$ such that $\models_1 A \to I$ and $\models_1 I \to B$ hold.
\end{definition}

\section{Skolemization}\label{skolemization}
We use skolemization to replace strong quantifiers in valid formulas such that the original formulas can be recovered. Note that several Skolem functions for the replacement of a single quantifier are necessary to represent proper suprema and proper infima.

\begin{definition}
Consider a formula $B$ in a context $A(B)$. Then its skolemization $A(sk(B))$ is defined as follows:

Replace all strong quantifier occurrences (positive occurrence of $\forall$ and negative occurrence of $\exists$) (note that no quantifiers in $A$ bind variables in $B$) of the form $\exists x C(x)$ (or $\forall x C(x)$) in $B$ by $\bigvee_{i = 1}^{|W|} C(f_i (\overline{x}))$ (or $\bigwedge_{i = 1}^{|W|}C(f_i (\overline{x}))$), where $f_i$ are new function symbols and $\overline{x}$ are the weakly quantified variables of the scope.

Skolem axioms are closed sentences
\[\forall \overline{x} (\exists y A(y, \overline{x}) \supset \bigvee_{i=1}^{|W|} A(f_{i}(\overline{x}), \overline{x}) \quad \mbox{and} \quad \forall \overline{x} (\bigwedge_{i=1}^{|W|} A(f_{i}(\overline{x}), \overline{x}) \supset \forall y A(y, \overline{x})) \]
where $f_i$ are new function symbols (Skolem functions).
\end{definition}

\begin{lemma}\label{beta}
\textbf{ }
\begin{enumerate}
\item If $\models^1 A(B)$ then $\models^1 A(sk(B)).$

\item If $S_1 \ldots S_k \models^1 A(sk(B))$ then $ S_1 \ldots S_k \models^1 A(B )$, for suitable Skolem axioms $S_1 \ldots S_k$.

\item If $ S_1 \ldots S_k \models^1 A$, where $S_1 \ldots S_k$ are Skolem axioms and $A$ does not contain Skolem functions then $ \models^1 A. $
\end{enumerate}
\end{lemma}
\begin{proof}
\textbf{ }

\begin{enumerate}
\item Note that
\[ \mbox{ if } \quad \models A(D) \quad \mbox{ then } \quad \models A(D \lor D) \]
and
\[ \mbox{ if } \quad \models A(D) \quad \mbox{ then } \quad \models A(D \land D). \]
Use Lemma \ref{alpha} and
\[ \models^1 D'(t) \supset \exists x D'(x), \quad \models^1 \forall x D'(x) \supset D'(t). \]

\item Use Lemma \ref{alpha} and suitable Skolem axioms to reconstruct strong quantifiers.
\item Assume $\not\models^1 A$. As usual, we have to extend the valuation to the Skolem functions to verify the Skolem axioms. There is a valuation in $\langle D_{\Phi^1}, \Omega_{\Phi^1} \rangle$ s.t. $\Phi^{1}(A) \not= 1$. Using at most $|W|$ Skolem functions and $AC$ we can always pick witnesses as values for the Skolem functions such that the first-order suprema and infima are reconstructed on the propositional level. ($AC$ is applied to sets of objects where the corresponding truth value is taken.)
\[ \supr\{\Phi^{1}(B(f_i (\overline{t}), \overline{t})) \, | \, 1 \leq i \leq |W| \} = \]
\[ \supr\{\Phi^{1}(B(d, \overline{t}) \, | \, d \in D_{\Phi^{1}} \} = \Phi^{1}(\exists y B(y, \overline{t})) \]
and
\[ \infi\{\Phi^{1}(B(f_i (\overline{t}), \overline{t})) \, | \, 1 \leq i \leq |W|\} = \]
\[ \infi \{\Phi^{1}(B(d, \overline{t})) \, | \, d \in D_{\Phi^{1}} \} = \Phi^{1}(\forall y B(y, \overline{t})). \]

\end{enumerate}
\end{proof}

\begin{example} We continue with the logic MC and its first-order variant $\MF$ introduced in Example \ref{ex.running}:
$\M = \textbf{L}^0 (\mathcal{L}^\to)$ is the set of valid propositional sentences and
$\MF = \textbf{L}^1(\mathcal{L}^\to)$ the set of valid first-order sentences. For the given logic $\MF$
\[ \exists x B(x) \supset sk(\exists y \forall z C(y,z)) \equiv \exists x B(x) \supset \exists y \bigwedge_{i = 1}^{5} C(y, f_i(y)). \]
\end{example}

\section{Expansions}\label{expansions}
Expansions, first introduced in \cite{miller1987compact}, are natural structures representing the instantiated variables for quantified formulas. They record the substitutions for quantifiers in an effort to recover a sound proof of the original formulation of Herbrand's Theorem. As we work with skolemized formulas, in this paper we we consider only expansions for formulas with weak quantifiers. Consequently the arguments are simplified.

In the following we assume that a constant $c$ is present in the language and that $t_1 , t_2, \ldots$ is a fixed ordering of all closed terms (terms not containing variables).

\begin{definition} A term structure is a structure $\langle D, \Omega \rangle$ such that $D$ is the set of all closed terms.
\end{definition} 

\begin{proposition} \label{prop.expansion} Let $\Phi^1 (\exists x A(x)) = \upsilon$ in a term structure. Then $\Phi^1 (\exists x A(x) = \Phi^1 (\bigvee_{i = 1}^{n} A(t_i ))$ for some $n$. Analogously for $\forall x A(x)$, i.e. let $\Phi^1 (\forall x A(x)) = \upsilon$ in a term structure, then $\Phi^1 (\forall x A(x)) = \Phi^1 (\bigwedge_{i=1}^{n} A(t_i ))$ for some $n$.
\end{proposition}
\begin{proof} Only finitely many truth values exists, therefore there is an $n$ such that the valuation becomes stable on $\bigvee_{i=1}^n A(t_i )$ ($\bigwedge_{i=1}^n A(t_i )$).
\end{proof}

\begin{definition} Let $E$ be a formula with weak quantifiers only. The $n$-th expansion $E_n$ of $E$ is obtained from $E$ by replacing inside out all subformulas $\exists x A(x)$ ($\forall x A(x)$) by $\bigvee_{i=1}^n A(t_i )$ ($\bigwedge_{i=1}^n A(t_i )$). $E_n$ is a Herbrand expansion iff $E_n$ is valid. In case there are only $m$ terms $E_{m+k} = E_m$.
\end{definition}

\begin{lemma} \label{lab.exp} Let $\Phi^1 (E) = \upsilon$ in a term structure. Then there is an $n$ such that for all $m \geq n$ $\Phi^1 (E_m ) = \upsilon$.
\end{lemma}
\begin{proof} We apply Proposition \ref{prop.expansion} outside in to replace subformulas $\exists x$ $A(x)$ ($\forall x$ $A(x)$) stepwise by $\bigvee_{i=1}^n$ $A(t_i )$ ($\bigwedge_{i=1}^n$ $A(t_i )$) without changing the truth value. The disjunctions and conjunctions can be extended to common maximal disjunctions and conjunctions.
\end{proof}

\begin{theorem} \label{th.expansion}
Let $E$ contain only weak quantifiers. Then $\models E$ iff there is a Herbrand expansion $E_n$ of $E$. 
\end{theorem}
\begin{proof}
$\Rightarrow$: Assume $\models E$ but $\not\models E_n$ for all $n$. Let $\Gamma_i = \{\Phi_{i,v}^0 | \Phi_{i,v}^0 (E_i) \not= 1\}$ and define $\Gamma = \bigcup \Gamma_i$. Note that the first index in $\Phi^0_{i,v}$ relates to the expansion level and the second index to all counter-valuations at this level. Assign a partial order $<$ to $\Gamma$ by $\Phi_{i,v}^0 < \Phi_{j,w}^0$ for $\Phi_{i,v}^0 \in \Gamma_i$, $\Phi_{j,w}^0 \in \Gamma_j$ and $i < j$ iff $\Phi_{i,v}^0$ and $\Phi_{j,w}^0$ coincide on the atoms of $E_i$. By K\"onig's Lemma there is an infinite branch $\Phi_{1,i_{1}}^0 < \Phi_{2, i_{2}}^0 < \ldots$. Define a term structure induced by an evaluation on atoms $P$:
\[ \Phi^1 (P) = \begin{cases} \upsilon \quad P \mbox{ occurs in some } E_n \mbox{ and } \Phi_{n, i_{n}} (P) = \upsilon \\ 1 \quad \mbox{else} \end{cases}\]
$\Phi^1 (E) \not= 1$ by Lemma \ref{lab.exp}.

$\Leftarrow$: Use Lemma \ref{alpha} and $\models A(t) \supset \exists x A(x)$ and $\models \forall x A(x) \supset A(t)$. Note that 
\[ \mbox{ if } \quad \models A(D \lor D) \quad \mbox{ then } \quad \models A(D) \]
and
\[ \mbox{ if } \quad \models A(D \land D) \quad \mbox{ then } \quad \models A(D). \]
\end{proof}

\begin{example} Consider the lattice in Example \ref{ex.running}, Fig. \ref{lattice} and the term ordering $c < d$. The expansion sequence of $P(c,d,d) \supset \exists x P(c,x,d)$ is 
\[ E_1 = P(c,d,d) \supset P(c,c,d), E_2 = P(c,d,d) \supset P(c,c,d) \lor P(c,d,d), E_{2+k} = E_2 .\]
The second formula is a Herbrand expansion.
\end{example}

\section{The Interpolation Theorem}\label{maintheorem}
\begin{theorem}
Interpolation holds for $\mathbf{L}^0(\mathcal{L}^\to)$ iff interpolation holds for
$\mathbf{L}^1(\mathcal{L}^\to)$.
\end{theorem}
\begin{proof}
\textbf{ }

$\Leftarrow$: trivial.

$\Rightarrow$: Assume $A \supset B$ is in the language of $\mathcal{L}^\to$ and $\models A \supset B$. 
\[ \models sk(A) \supset sk(B) \quad \mbox{ by Lemma \ref{beta} 1. } \]
Construct a Herbrand expansion $A_H \supset B_H$ of $sk(A) \supset sk(B)$ by Theorem \ref{th.expansion}.\newline
Construct the propositional interpolant  $I^*$ of $A_H \supset B_H$, 
\[ \models A_H \supset I^* \quad \mbox{and } \quad \models I^* \supset B_H . \]
Use Lemma \ref{alpha} and 
\[ \models A(t) \supset \exists x A(x), \quad \models \forall x A(x) \supset A(t) \]
to obtain
\[ \models sk(A) \supset I^* \quad \mbox{and} \quad \models I^* \supset sk(B) \]
Order all terms $f(t)$ in $I^*$ by inclusion where $f$ is not in the common language. Let $f^{*}(\overline{t})$ be the maximal term.
\begin{enumerate}
\item[i.] $f^*$ is not in $sk(A)$. Replace $f^{*}(\overline{t})$ by a fresh variable $x$ to obtain
\[ \models sk(A) \supset I^{*} \{x / f^{*}(\overline{t})\}. \] 
But then also
\[ \models sk(A) \supset \forall x I^{*} \{x / f^{*}(\overline{t})\} \]
and
\[ \models \forall x I^{*} \{x / f^{*}(\overline{t})\} \supset sk(B) \] 
by
\[ \models \forall x I^{*} \{x / f^{*}(\overline{t})\} \supset I^{*}.\]
\item[ii.] $f^*$ is not in $sk(B)$. Replace $f^{*}(\overline{t})$ by a fresh variable $x$ to obtain
\[ \models I^{*} \{x / f^{*}(\overline{t})\} \supset sk(B). \]
But then also
\[ \models \exists x I^{*} \{x / f^{*}(\overline{t})\} \supset sk(B) \]
and
\[ \models sk(A) \supset \exists x I^{*} \{x / f^{*}(\overline{t})\} \]
by
\[ \models I^* \to \exists x I^{*} \{x \backslash f^{*}(\overline{t})\}.\]
\end{enumerate}
Repeat this procedure till all functions and constants not in the common language (among them the Skolem functions) are eliminated from the middle formula. Let $I$ be the result. $I$ is an interpolant of $sk(A) \supset sk(B)$. By Lemma \ref{beta} 2,3 $I$ is an interpolant of $A \to B$. For a similar construction for classical first-order logic see Chapter 8.2 of \cite{baaz2011methods}.
\end{proof}

\begin{corollary}  \label{cor.1}
If interpolation holds for $\mathbf{L}^0(\mathcal{L}^\to)$, $\quad \models A \to B$ and $A \to B$ contains only weak quantifiers, then there is a quantifier-free weak interpolant with common predicates for $A \to B$.
\end{corollary}
\begin{remark}
Corollary \ref{cor.1} cannot be strengthened to provide a quantifier-free interpolant with common predicate symbols and common function symbols for $A \to B$. Consider 
\[ Q_\forall A(x_1 , f_1(x_1), x_2 , f_2 (x_1 , x_2) , \ldots ) \to Q_\exists A(g_1 , y_1 , g_2 (y_1), y_2 , g_3 (y_1 , y_2), \ldots), \]
where $Q_\forall = \forall x_1 \forall x_2  \ldots$ and $Q_\exists = \exists y_1 \exists y_2 \ldots$. This is the skolemization of
\[ \forall x_1 \exists x'_1 \forall x_2 \ldots A(x_1 , x'_1, x_2 , \ldots ) \to \forall x_1 \exists x'_1 \forall x_2 \ldots A(x_1 , x'_1, x_2 , \ldots ),\]
where $\forall x_1 \exists x'_1 \forall x_2 \exists x'_2 \ldots A(x_1 , x'_1, x_2 , x'_2, \ldots )$ is the only possible interpolant modulo provable equivalence with common predicate and function symbols.
\end{remark}

\begin{example} \label{ex.interpolation} Example \ref{ex.running} continued. For the given logic we calculate the interpolant for
\[ \exists x (B(x) \land \forall y C(y)) \to \exists x (A(x) \lor B(x)). \]
\begin{enumerate}
\item Skolemization
\[ \bigvee_{i=1}^{5}(B(c_i) \land \forall y C(y)) \to \exists x (A(x) \lor B(x)). \]
\item Herbrand expansion
\[ \bigvee_{i=1}^{5}(B(c_i) \land C(c_1 )) \to \bigvee_{i=1}^{5} (A(c_i ) \lor B(c_i )). \]
\item Propositional interpolant
\[ \bigvee_{i=1}^{5}(B(c_i) \land C(c_1 )) \to \bigvee_{i=1}^{5}B(c_i ) \quad \bigvee_{i=1}^{5}B(c_i ) \to \bigvee_{i=1}^{5}(A(c_i ) \lor B(c_i )). \]
\item Back to the Skolem form
\[ \bigvee_{i=1}^{5}(B(c_i) \land \forall y C(y)) \to \bigvee_{i=1}^{5}B(c_i ) \quad \bigvee_{i=1}^{5}B(c_i ) \to \exists x (A(x) \lor B(x)). \]
\item Elimination of function symbols and constants not in the common language from $\bigvee_{i=1}^{5}B(c_i )$. Result:
\[ \exists z_1 \ldots \exists z_5 \bigvee B(z_i ). \]
\item Use the Skolem axiom
\[ \exists x (B(x) \land \forall y C(y)) \to \bigvee_{i=1}^{5} B(c_i ) \land \forall y C(y) \]
to reconstruct the original first-order form.
\item The Skolem axiom can be deleted.
\end{enumerate}
\end{example}

\begin{example} \label{ex.10}
$\mathcal{L}^\to = \langle \{0,1\}, \leq, \lor, \land, \to \rangle$ be the lattice of classical logic.
\[ \mbox{SPECTRUM}(\mathcal{L}^\to) = \{\{0\}, \{1\}, \{0,1\}\} \]
This is the maximal possible spectrum by Proposition \ref{prop.c} $1$.

Let $\mathcal{L'}^\to$ and $\mathcal{L''}^\to$ be the extension of $\mathcal{L}^\to$ by $\{\overline{0}\}$ and $\{\overline{0}, \overline{1}\}$, respectively. $\mathbf{L^1}(\mathcal{L'}^\to)$ and $\mathbf{L^1}(\mathcal{L''}^\to)$ interpolate as all truth constants are representable by closed formulas. This is Craig's result, which does however not cover $\mathbf{L^1}(\mathcal{L'''}^\to)$, where $\mathcal{L'''}^\to$ is the extension of $\mathcal{L}^\to$ with $\{\overline{1}\}$. We have only to show that $\mathbf{L^0}(\mathcal{L'''}^\to)$ interpolates.

First note in general that 
\[ \bigvee_i E_i \to \bigwedge_j F_j \]
interpolates iff there are interpolants
\[ E_i \to I_{ij} \quad I_{ij} \to F_j. \]
$\bigwedge_j \bigvee_i I_{ij}$ is a suitable interpolant. Now use the value preserving transformations
\[ D(A \land B \to C) \quad \Leftrightarrow \quad D(A \to C \lor B \to C) \]
\[ D(A \lor B \to C) \quad \Leftrightarrow \quad D(A \to C \land B \to C) \]
\[ D((A \to B) \to C) \quad \Leftrightarrow \quad D(C \lor (A \land (B \to C))) \]
\[ D(x) \quad \Leftrightarrow \quad D(\top \to x) \]
for variables $x$ together with distributions and simplifications, to reduce the problem to 
\[ \bigwedge_i (u_i \to v_i ) \to \bigvee_j (s_j \to t_j ) \]
$v_i , t_j$ variables, $u_i , s_j$ variables or $\top$. We assume that the succedent is not valid (otherwise $\top$ is the interpolant). So any variable occurs either in the $s_j$ group or in the $t_j$ group. Close the antecedents under transitivity of $\to$. There is a common implication $u \to v$, an interpolant (Otherwise there is a counter valuation by assigning $0$ to all $t_j$ and extending this assignment in the antecedent such that if $v_i$ is assigned $0$ also $u_i$ is assigned $0$. No $s_j$ is assigned $0$ by this procedure. Assign $1$ to all other variables and derive a contradiction to the assumption, that the initial implication is valid). Therefore, $\mathbf{L^1}(\mathcal{L'''}^\to)$ interpolates.
\end{example}

\begin{example} $n$-valued G\"odel logics. \\
A finitely-valued G\"odel logic extended by constants interpolates if there are no consecutive two truth values different to $0,1$ not expressible by closed terms (see Theorem $11$ in \cite{DBLP:journals/aml/BaazV99}). Note that $\overline{0}$ is able to express $\overline{1}$ and vice versa but no other truth constant is expressible by any other truth constants.

Let $\mathcal{G}_n = \langle W_n , \lor, \land, \to, \overline{0} \rangle$, where $\overline{0} = 0$ and $W_n = \{ 0, \frac{1}{n-1}, \ldots, \frac{n-2}{n-1}, 1 \}$.
\[ u \to v = \begin{cases} 1 \quad u \leq v \\ v \quad \mbox{else} \end{cases} \]
Let $T_n$ be the set of non-empty subsets of $W_n$ not containing $0,1$ such that besides $0,1$ no two consecutive truth values lack.
$$ \mbox{SPECTRUM}(\mathcal{G}_n) = \{ \Gamma \ | \ \Delta \leq \Gamma \mbox{ for some } \Delta \in T_n\}.$$
\end{example}

\begin{example} Finitely-valued Łukasiewicz logic.\\
Let Ł$_n = \langle W_n, \lor, \land, \to, \overline{0} \rangle$, where $W_n = \{ a_1, \ldots, a_n \}$, $a_i \leq a_j \Leftrightarrow i \leq j$, $\overline{0} = a_1$. Note that Ł$_n$ interpolates iff all truth values are representable (see Theorem $17$ in \cite{DBLP:journals/aml/BaazV99}). The first-order SPECTRUM is given by the propositional SPECTRUM.
Let $T_n$ be the set of non-empty subsets of $W_n$ such that the greatest common divisor of the indices of elements of the set is $1$. The SPECTRUM of Ł$_n$ is the set of all supersets of $T_n$. 
\end{example}

\section{Conclusion}\label{conclusion}
Extending the notion of expansion to formulas containing strong quantifiers might be possible to cover logics which do not admit skolemization, e.g. logics based on non-constant domain Kripke frames (such notions of expansion are in the spirit of Herbrand's original proof of Herbrand's Theorem). \cite{DBLP:conf/frocos/BaazL17} contains an application to the prenex fragment of first-order G\"odel logic.

Another possibility is to develop unusual skolemizations e.g. based on existence assumptions \cite{baaz2006skolemization} or on the addition of Skolem predicates instead of Skolem functions as in \cite{godel1930vollstandigkeit}.

The methodology of this paper can also be used to obtain negative results. First-order $S5$ does not interpolate by a well-known result of Fine \cite{fine1979failures}. As propositional $S5$ interpolates, first-order $S5$ cannot admit skolemization together with expansions in general.

\bibliographystyle{plain}
\bibliography{refs}

\end{document}